\documentclass[draft]{article}  
\usepackage[final]{showkeys}
\usepackage[breaklinks, final]{hyperref}
\usepackage{amsmath, amssymb, amsthm}
\usepackage{float, tikz, caption, subcaption}
\usetikzlibrary{calc, arrows, positioning}
\usepackage{tabu, enumitem, mathtools}

\newcommand{\m}[1]{\textbf{\uppercase{#1}} }

\newcommand{\vr}[1]{\mathcal{\uppercase{#1}} }

\newcommand{\usub}{\subseteq}

\theoremstyle{plain} \newtheorem{thm}{Theorem}[section]
\newtheorem*{thm*}{Theorem}
\newtheorem{lm}[thm]{Lemma}
\newtheorem{observation}[thm]{Observation}
\newtheorem{cor}[thm]{Corollary}
\newtheorem*{cor*}{Corollary}
\theoremstyle{remark} 
\theoremstyle{definition} \newtheorem{df}[thm]{Definition}
\newtheorem{prob}[thm]{Problem}

\numberwithin{equation}{section}      
\newcommand{\pmat}[1]{ \ensuremath{\begin{pmatrix} #1 \end{pmatrix}} }

\DeclareMathOperator{\Jchain}{\text{J}}
\DeclareMathOperator{\DJchain}{\text{DJ}}
\DeclareMathOperator{\Gchain}{\text{G}}
\DeclareMathOperator{\DGchain}{\text{DG}}
\DeclareMathOperator{\Pchain}{\text{P}}

\DeclareMathOperator{\Erel}{\mathit{E}}
\DeclareMathOperator{\Frel}{\mathit{F}}

\begin{document}
\title{Absorption and directed J\'{o}nsson terms}
\author{A. Kazda,
  M. Kozik\footnote{The second author was supported by National Science
    Center grant DEC-2011/01/B/ST6/01006.},
  R. McKenzie,
  M. Moore }
\date{\today}
\maketitle


\begin{abstract}
We prove that every congruence distributive variety has directed
J\'{o}n\-sson terms, and every congruence modular variety has directed Gumm
terms. The directed terms we construct witness every case of absorption
witnessed by the original J\'{o}nsson or Gumm terms. This result is
equivalent to a pair of claims about absorption for admissible preorders in
CD and CM varieties, respectively. For finite algebras, these absorption
theorems have already seen significant applications, but until now, it was
not clear if the theorems hold for general algebras as well. Our method also
yields a novel proof of a result by P. Lipparini about the existence a chain
of terms (which we call Pixley terms) in varieties that are at the same time
congruence distributive and $k$-permutable for some $k$.
\end{abstract}

\section{Introduction}  
In 1967, Bjarni J\'{o}nsson~\cite{jonsson} proved that a variety $\vr v$ is
congruence distributive~(CD) if and only if it has, for some $n$, a sequence
of terms $J_0(x,y,z)$, $\dots$, $J_n(x,y,z)$ satisfying a certain system of
equations, namely, $J_0(x,y,z)=x$, $J_n(x,y,z)=z$, $J_i(x,y,x)=x$ for each
$0\leq i\leq n$, and for each $0\leq i<n$, either the equation
$J_i(x,x,y)=J_{i+1}(x,x,y)$ or the equation $J_i(x,y,y)=J_{i+1}(x,y,y)$.
This {Maltsev} condition can be formulated more specifically in several
equivalent ways. The following formulation is convenient for our purposes:
for some $n\geq 0$ and terms $J_0(x,y,z), \dots, J_{2n+1}(x,y,z)$, consider
the system of equations
\begin{equation} \tag{$\Jchain(n)$} \begin{aligned}
  & J_1(x,x,y) = x, \qquad J_{2n+1}(x,y,y) = y, \\
  & \begin{aligned}
    & J_i(x,y,x) = x,                   & \quad & \text{for } 0\leq i\leq 2n+1, \\
    & J_{2i+1}(x,y,y) = J_{2i+2}(x,y,y) & \quad & \text{for } 0\leq i\leq n-1, \\
    & J_{2i}(x,x,y) = J_{2i+1}(x,x,y)   & \quad & \text{for } 1\leq i\leq n,
  \end{aligned}
\end{aligned} \end{equation}
and call this package of equations $\Jchain(n)$. By a chain of
\emph{J\'{o}nsson terms} for a variety $\vr v$, we mean a sequence of terms
satisfying over $\vr v$ the equations $\Jchain(n)$ for some $n$. J\'{o}nsson
proved that an algebra $\m a$ has terms obeying the equations $\Jchain(n)$,
for some $n$, if and only if the congruence lattice of every algebra in the
variety generated by $\m a$ is distributive. A system of \emph{directed
J\'{o}nsson terms} for $\vr v$ consists, for some $n\geq 1$, of terms
$D_1(x,y,z), \dots, D_n(x,y,z)$ satisfying over $\vr v$ the equations
$\DJchain(n)$:
\begin{equation} \tag{$\DJchain(n)$} \begin{aligned}
  & D_1(x,x,y) = x, \qquad D_n(x,y,y) = y, \\
  & \begin{aligned}
    & D_i(x,y,x) = x              & \quad & \text{for } 1\leq i\leq n, \\
    & D_i(x,y,y) = D_{i+1}(x,x,y) & \quad & \text{for } 1\leq i<n.
  \end{aligned}
\end{aligned} \end{equation}
Our chief purpose is to show that a variety has J\'{o}nsson terms if and
only if it has directed J\'{o}nsson terms. Moreover, in such a case, one can
find a sequence of terms which satisfy $\Jchain(n)$ and $\DJchain(2n+1)$ for
some $n$ at the same time. These two results are contained in
Corollary~\ref{cor:J} and Observation~\ref{obs:bestofboth}.

H.P. Gumm~\cite{gumm} proved that a variety $\vr v$ is congruence modular
(CM) if and only if it has, for some $n\geq 0$, a sequence of terms
$J_1(x,y,z), \dots, J_{2n+1}(x,y,z)$, and $P(x,y,z)$ satisfying the
equations $\Gchain(n)$:
\begin{equation} \tag{$\Gchain(n)$} \begin{aligned}
  & J_1(x,x,y) = x, \qquad J_{2n+1}(x,y,y) = P(x,y,y), \qquad P(x,x,y) = y, \\
  & \begin{aligned}
    & J_i(x,y,x)=x                      & \quad & \text{for } 0\leq i\leq 2n+1, \\
    & J_{2i+1}(x,y,y) = J_{2i+2}(x,y,y) & \quad & \text{for } 0\leq i\leq n-1, \\
    & J_{2i}(x,x,y) = J_{2i+1}(x,x,y)   & \quad & \text{for } 1\leq i\leq n.
  \end{aligned}
\end{aligned} \end{equation}
\emph{Directed Gumm terms} are terms $D_1(x,y,z), \dots, D_n(x,y,z)$, and
$Q(x,y,z)$ satisfying $\DGchain(n)$ for some $n\geq 1$:
\begin{equation} \tag{$\DGchain(n)$} \begin{aligned}
  & D_1(x,x,y) = x, \qquad D_n(x,y,y) = Q(x,y,y), \qquad Q(x,x,y) = y, \\
  & \begin{aligned}
    & D_i(x,y,x) = x              & \quad & \text{for } 1\leq i\leq n, \\
    & D_i(x,y,y) = D_{i+1}(x,x,y) & \quad & \text{for } 1\leq i<n.
  \end{aligned}
\end{aligned} \end{equation}

Similarly to the congruence distributive case, we show that a variety has
Gumm terms if and only if it has directed Gumm terms, and that given Gumm
terms we can find terms satisfying $\Gchain(n)$ and $\DGchain(2n+1)$ for
some $n$ at the same time.  These two results are contained in
Theorem~\ref{thm1} and Observation~\ref{obs:bestofboth}.

Our context makes it natural to introduce another Maltsev condition that
looks similar to directed J\'{o}nsson terms but is actually much stronger.
The condition is that for some $n\geq 1$ there are terms $P_1(x,y,z), \dots,
P_n(x,y,z)$ satisfying $\Pchain(n)$:
\begin{equation} \tag{$\Pchain(n)$} \begin{aligned}
  & P_1(x,y,y) = x, \qquad P_n(x,x,y) = y, \\
  & \begin{aligned}
    & P_i(x,y,x) = x              & \quad & \text{for } 1\leq i\leq n, \\
    & P_i(x,x,y) = P_{i+1}(x,y,y) & \quad & \text{for } 1\leq i<n.
  \end{aligned}
\end{aligned} \end{equation}
This condition, which we call \emph{Pixley terms}, first appeared in P.
Lipparini~\cite{lipparini}.

Observe that if we remove the equations ``$J_i(x,y,x)=x$'' from
$\DJchain(n)$, we obtain a Maltsev condition that is always trivially
satisfied by taking $D_1(x,y,z)=y$ and $D_i(x,y,z)=z$ for all $1<i\leq n$.
For contrast, removing the equations $P_i(x,y,x)=x$ from $\Pchain(n)$
produces the classical Hagemann-Mitschke terms~\cite{hm}, and these have
highly nontrivial consequences. A variety has a chain of $n$
Hagemann-Mitschke terms if and only if it has $(n+1)$-permuting congruences.
The variety of lattices, for example, satisfies $\Jchain(1)$ but does not
have Hagemann-Mitschke terms.

A. Pixley~\cite{pixley} proved that a variety is congruence distributive and
all its congruences permute if and only if it satisfies $\Pchain(1)$. A term
$P_1(x,y,z)$ for which 
\[
  P_1(x,y,x) = P_1(x,y,y) = P_1(y,y,x) = x
\]
holds has long been called a \emph{Pixley term}. In this connection, note
that the term $J_1(x,y,z)$ with the equations $J_1(x,y,x) = J_1(x,x,y) =
J_1(y,x,x) = x$ constituting J\'{o}nsson terms $\Jchain(0)$ is familiarly
known as a \emph{majority} term; and both $\Jchain(0)$ and $\DJchain(1)$ are
just asserting that we have a majority term.

\bigskip
Here is our principal result about these Maltsev conditions.
\begin{thm}\label{thm:main1}
Let $\vr v$ be any variety of algebras.
\begin{enumerate}
  \item\label{CDpart} $\vr v$ is congruence distributive if and only if it
    has directed J\'{o}nsson terms. In such a case there is a sequence of
    terms satisfying $\DJchain(2n+1)$ and $\Jchain(n)$ at the same time~(for
    some $n\geq 1$). See Corollary~\ref{cor:J} and
    Observation~\ref{obs:bestofboth}.
\item\label{nCPpart} For any integer $k\geq 1$, a variety $\vr v$ is
  congruence distributive and has $(k+1)$-permuting congruences if and only
  if it satisfies $\Pchain(k)$. See Theorem~\ref{thm2} for the
  ``$\Rightarrow$'' implication.
\item\label{CMpart} $\vr v$ is congruence modular if and only if it has
  directed Gumm terms. In such a case there is a sequence of terms
  satisfying $\Gchain(n)$ and $\DGchain(2n+1)$ at the same time~(for some
  $n\geq 1$). See Theorem~\ref{thm1} and Observation~\ref{obs:bestofboth}.
\end{enumerate}
\end{thm}
Statement (2) is Proposition 5 in P. Lipparini~\cite{lipparini}. However,
our proof, given in Section~\ref{section6}, is new, and shows more.

\begin{observation}\label{obs:bestofboth}  
Let $\vr v$ be a variety that admits a chain of terms satisfying
$\DJchain(n)$. Then $\vr v$ admits a chain of terms that satisfy
$\Jchain(n-1)$ and $\DJchain(2n-1)$ at the same time. Similarly,
$\DGchain(n)$ implies the existence of a chain of terms that simultaneously
satisfies $\Gchain(n-1)$ and $\DGchain(2n-1)$.
\end{observation}
\begin{proof}
Given directed J\'{o}nsson terms $D_1,\dots,D_n$, we produce the new terms
by letting
\begin{gather*}
  J_1(x,y,z) = D_1(x,y,z), \qquad J_{2i}(x,y,z) = D_{i+1}(x,x,z), \\
  J_{2i+1}(x,y,z) = D_{i+1}(x,y,z) \qquad \text{for } 1\leq i\leq n-1.
\end{gather*}
We leave to the reader the easy proof that $\DGchain(n)$ implies appropriate
terms for congruence modular varieties.
\end{proof}  

Similarly, $\Pchain(n)$ implies $\Jchain(n)$: given some Pixley terms
$P_1,\dots,P_n$, take 
\begin{align*}
  & J_1(x,y,z) = x, \qquad J_{2n+1}(x,y,z) = z, \\
  & \begin{aligned}
    & J_{2i}(x,y,z) = P_i(x,y,z)       & \quad & \text{for } 0\leq i<n, \\
    & J_{2i+1}(x,y,z) = P_{i+1}(x,z,z) & \quad & \text{for } 0\leq i<n.
  \end{aligned}
\end{align*}
It is an easy exercise to show that $\Pchain(k)$ implies $(k+1)$-permuting
congruences.

Our proof of the converse implications, that is, $\Jchain(n)$ implies
$\DJchain(k)$ for some $k$, and $\Gchain(n)$ implies $\DGchain(k)$ for some
$k$, will take some work and will be concluded in Sections~\ref{section4}
and \ref{section5}. The fact that a $(k+1)$-permutable variety with
J\'{o}nsson terms satisfies $\Pchain(k)$ is demonstrated in
Section~\ref{section6}.

\section{Absorption}\label{section2}   
The notion of \emph{absorption} was introduced by L. Barto and M. Kozik
\cite{BK12}, who proved deep results about absorption in finite algebras and
used this theory as a powerful tool for applying universal algebraic methods
in the study of constraint satisfaction problems (this area where universal
algebra and theoretical computer science meet has blossomed over the past
decade).

If $\m C$ and $\m D$ are subalgebras of an algebra $\m a$ we say that $\m C$
\emph{absorbs} $\m D$ if $\emptyset \neq C\usub D$ and there is a term
operation $s(x_1,\dots,x_n)$ of the algebra $\m a$ such that $\m a\models
s(x,\dots,x)=x$ (i.e. $s$ is idempotent) and whenever $\bar{d}\in D^n$ with
$d_i\in D\setminus C$ for at most one $i\in \{1,\dots,n\}$, then
$s(\bar{d})\in C$. We denote the fact that $\m C$ absorbs $\m D$ in this
sense by $\m C\lhd \m D$, or $\m C\lhd_s \m D$ where $s$ is the term
operation that witnesses the absorption.

In this paper, however, a different variant of absorption is needed. We will
say that a sequence $J_1,\dots,J_{2n+1}$ of terms is a chain of \emph{weak
J\'{o}nsson terms} if $J_1,\dots,J_{2n+1}$ satisfy all of the equations
$\Jchain(n)$ except perhaps $J_i(x,y,x)=x$.  We define \emph{weak directed
J\'{o}nsson chains}, \emph{weak Gumm chains}, and \emph{weak directed Gumm
chains} similarly, always dropping the requirement that $J_i(x,y,x)=x$.

If $\m C$ and $\m D$ are subalgebras of $\m a$, $\emptyset\neq C\usub D$,
and $t(x,y,z)$ is a ternary idempotent term operation of $\m a$, then we
write $\m C\lhd^m_t \m D$ and say that $\m C$ \emph{middle absorbs $\m D$
with respect to $t$} if $t(a,b,c)\in C$ whenever $a,c\in C$ and $b\in D$.
If $\vr t$ is a set of ternary idempotent term operations of $\m a$, we say
that $\m C$ middle absorbs $\m D$ with respect to $\vr t$, written $\m
C\lhd^m_{\vr t} \m D$, provided that $\m C\lhd^m_t \m D$ for every $t\in \vr
t$.

We are interested in four special cases of middle absorption: J\'{o}nsson
absorption, Gumm absorption, and directed versions thereof. We save Gumm
absorption for the end of this paper and concentrate on J\'{o}nsson
absorption for now.

We say that $\m C$ \emph{J\'{o}nsson absorbs} $\m D$ if $\m C\lhd^m_{\vr j}
\m D$, where $\vr j$ is a sequence of weak J\'onsson terms. Directed
J\'{o}nsson absorption is defined analogously, with weak directed J\'onsson
terms. We shall write $\m C\lhd_J \m D$ (in words, $C$ J\'{o}nsson absorbs
$D$) to indicate either that $\m C \lhd^m_{\vr j}\m D $ for some chain $\vr
j$ of weak J\'onsson terms, or that $\m C\lhd^m_{\vr j}\m D$ for a specific
system of terms that is being held fixed. The context will make clear which
is meant. Our use of the notation $\m C\lhd_{DJ} \m D$ (directed J\'{o}nsson
absorption) is analogous.

It is an easy exercise to show that if $\m a$ is a finite idempotent algebra
then $\m a$ admits a chain of J\'onsson terms (respectively, directed
J\'onsson terms) if and only if for every $a\in A$ we have $\{a\}\lhd_J \m
a$ (respectively, $\{a\}\lhd_{DJ} \m a$). Moreover, it is immediate that
\nocite{barto-kazda} standard absorption, $\m C\lhd \m D$, implies $\m
C\lhd_{DJ} \m D$, which in turn implies $\m C\lhd_{J} \m D$. Indeed, suppose
that $\m C\lhd_t \m D$ for $t=t(x_1,\dots, x_n)$. Take
$Q_1(x,y,z)=t(x,\dots,x,y)$,
\[
  Q_j(x,y,z)=t(x,\dots,x,y,z,\dots,z)
  \qquad \text{with $y$ in the $(n-j+1)$-th place},
\]
for $1<j<n$, and $Q_n(x,y,z)=t(y,z,\dots,z)$. This is a system of directed
J\'{o}nsson operations with respect to which $\m C$ middle absorbs $\m D$.
The proof that $\m C \lhd_{DJ} \m D$ implies $\m C \lhd_{J} \m D$ is similar
to the argument that if $\vr v$ is a variety with a chain of terms that
satisfy $\DJchain(n)$, then $\vr v$ has a chain of terms satisfying
$\Jchain(n-1)$.

The second principal result of our paper is included in
Theorem~\ref{thm:main3}. Before introducing it we present a proof of the
same result for finite algebras. The result was motivated by
Barto~\cite{barto} and the proof essentially follows the argument presented
there.

\begin{thm}\label{thm:preorders}   
Suppose that $E$ and $F$ are admissible preorders on $\m a$ (that is, they
are subalgebras of $\m a^2$ that are reflexive and transitive over $A$). If
$E\lhd_{J} F$, then $E=F$.
\end{thm}
\begin{proof}[Proof (assuming $\m a$ is finite)]
Suppose that $E$ and $F$ are admissible preorders of the finite algebra $\m
a$ and $E\lhd_{J} F$. Let $J_1,\dots,J_{2n+1}$ be the terms that witness the
J\'onsson absorption, and let $(a,b)\in F$. We must show that $(a, b)\in E$.
For ease of notation, we will write $x\dashrightarrow y$ for $(x,y)\in F$
and $x\to y$ for $(x,y)\in E$ (so we want to show $(a\dashrightarrow
b)\Rightarrow (a\to b)$).

Without loss of generality, we can assume that $\m a$ is generated by
$\{a,b\}$ so that $b$ is a top element in the order $\dashrightarrow$, and
since $\m A$ is finite we can assume that $a$ is $\to$-maximal in $\m a$ (if
there was a $c$ strictly $\to$-larger than $a$ in the algebra generated by
$\{a,b\}$, we could replace $a$ by $c$). Using
$J_1(\to,\dashrightarrow,\to)\subset \to$ and a J\'{o}nsson equation, we
have 
\[
  a=J_1(a,a,b)\to J_1(a,b,b).
\]
Now we prove by induction on $i$ that $a\to J_{2i+1}(a,b,b)$ for all $0\leq
i\leq n$. Suppose that $a\to J_{2i+1}(a,b,b)=J_{2i+2}(a,b,b)=q$. Let
$p=J_{2i+2}(a,a,b)$. Absorption gives that $p\to q$, and that $p=J_{2i+3}(a,a,b) \to
J_{2i+3}(a,b,b)$, so all we need to show is that $a\to p$.

The maximality of $a$ yields $q\to a$. Since $p$ lies in the subalgebra
generated by $\{a,b\}$, we have $a\dashrightarrow p$. Putting it together, we
have $q\to a \dashrightarrow p\to q$.

\begin{figure}\centering  
  \begin{tikzpicture}[font=\scriptsize, smooth, node distance=2cm,
    shorten <=3pt, shorten >=3pt ]

    \node (b){$b$};
    \node[below of=b, node distance=4cm]  (a){$a$};
    \node[right of=a] (q){};
    \node[right of=b] (J){$J_{2i+3}(a,b,b)$};
    \node[above of=q](p){};
    \node[right of=p] (pname){$p=J_{2i+2}(a,a,b)=J_{2i+3}(a,a,b)$};
    \node[right of=q] (qname){$q=J_{2i+1}(a,b,b)=J_{2i+2}(a,b,b)$};
    \draw
    (a) edge[dashed, ->] (b)
    (a) edge[dashed, ->] (p)
    (p) edge[dashed, ->, bend left] (q)
    (q) edge[dashed, ->, bend left] (p)
    (q) edge[->, bend left] (a)
    (a) edge[->, bend left] (q)
    (q) edge[->] (p)
    (p) edge[->] (J);
  \end{tikzpicture}
  \caption{The elements $a,b,p,q$ in the finite case of
    Theorem~\ref{thm:preorders}.}\label{figFinite} 
\end{figure}
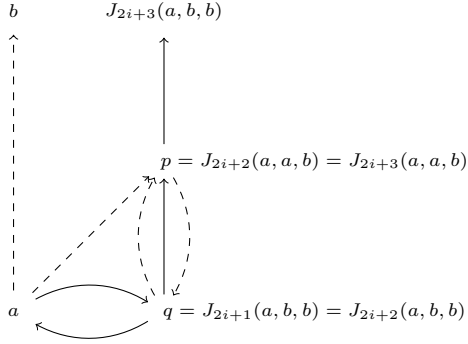 

We have obtained  $q\dashrightarrow p\dashrightarrow q$. Absorption now
allows us to prove that $q\to p$:
\begin{multline*}
  q
    = J_1(q,q,p)
    \to J_1(q,p,p)
    = J_2(q,p,p)
    \to J_2(q,q,p) \\
  = J_3(q,q,p)
    \to \dots 
    \to J_{2n+1}(q,q,p)
    \to J_{2n+1}(q,p,p)
    =p.
\end{multline*}
Therefore, $a\to q\to p\to J_{2i+3}(a,b,b)$ (see Figure~\ref{figFinite} as a
reference to what we did) and we have $a\to J_{2i+1}(a,b,b)$ for all $i$. In
particular, $a\to J_{2n+1}(a,b,b)=b$, and we are done.
\end{proof}  

Note that there is a straightforward proof of the conclusion of the above
Theorem if we assume that $E\lhd_{DJ} F$ instead of $E\lhd_{J} F$.

Using Theorem~\ref{thm:preorders}, we will now prove part~\ref{CDpart} of
Theorem~\ref{thm:main1} in the finite case. Let $\vr v$ be an idempotent CD
variety, and let $\m F_2(x,z)$ and $\m F_3(x,y,z)$ be the free two and three
generated algebras in $\vr v$. Let
\begin{gather*}
  \vr g=\big\{ t(x,y,z)\in \m F_3 
    \colon \text{$t(x,y,x)= x$ holds in $\vr v$} \big\}, \\
  F = \{( t(x,x,z),t(x,z,z) ) \colon t\in \m F_3\},
  \quad \text{and} \quad
  E = \{( t(x,x,z),t(x,z,z) ) \colon t\in \vr g\}.
\end{gather*}
Denote by $\to$ (resp. $\dashrightarrow$) the transitive closures of $E$
(resp. $F$). It is straightforward to show that $E$, $F$, $\to$, and
$\dashrightarrow$ are all admissible relations on $\m F_2$. Since $E$, $F$
are reflexive, the relations $\to$ and $\dashrightarrow$ are preorders on
$\m F_2$.

Observe that $(x,z)\in F$ (we can choose $t$ to be the projection to the
second coordinate). Let $\vr j$ be a chain of J\'onsson terms in $\vr v$.
One can easily verify that then $E\lhd^m_{\vr j} F$, from which it follows
that $\to \lhd^m_{\vr j} \dashrightarrow$. Using
Theorem~\ref{thm:preorders}, we then have that $\to$ and $\dashrightarrow$
are the same. In particular, $x\to z$, and there is a sequence of terms
$D_1,\dots,D_m\in \vr g$ witnessing this fact. Examining the terms
$D_1,\dots, D_m$, we get the following system of equalities in $\vr v$:
\begin{align*}
  & D_1(x,x,z) = x, \qquad D_m(x,z,z) = z, \\
  & \begin{aligned}
    & D_i(x,y,x) = x              & \quad & \text{for } 1\leq i\leq m, \\
    & D_i(x,z,z) = D_{i+1}(x,x,z) & \quad & \text{for } 1\leq i<m,
  \end{aligned}
\end{align*}
which means that $D_1,\dots,D_m$ are directed J\'onsson terms.

Of course, the sequence of proofs presented so far only works when $\m F_2$
is finite, but we will improve that. In fact, we will show that one can
always make J\'onsson absorption into directed J\'onsson absorption.
\begin{thm}\label{thm:main3}
Let $\vr v$ be a variety, and $\vr j$ be a chain of weak J\'onsson terms of
$\vr v$. Then there exists a chain $\vr {d}$ of weak \emph{directed}
J\'onsson terms of $\vr v$ such that for all $\m A,\m B\in \vr v$ we have
$\m B\lhd^m_{\vr j} \m A \Rightarrow \m B \lhd^m_{\vr{d}} \m A$.
\end{thm}
The proof of Theorem~\ref{thm:main3} will have to wait until
Section~\ref{section4}, after we have constructed suitable tools.

\section{Paths in the free algebra}\label{section3}  
This section contains the core of this paper -- a proof of a somewhat
technical result from which Theorem~\ref{thm:main3} follows.

We choose and fix a variety $\vr w$ whose only basic operations are
$J_1\dots, J_{2k+1}$, which satisfy the equations
\begin{equation} \begin{aligned}
  & J_1(x,x,y) = x, \\
  & \begin{aligned}
    & J_{2i+1}(x,y,y) = J_{2i+2}(x,y,y) & \quad & \text{for } 0\leq i\leq k-1, \\
    & J_{2i}(x,x,y) = J_{2i+1}(x,x,y)   & \quad & \text{for } 1\leq i\leq k.
  \end{aligned}
\end{aligned} \end{equation}
By adding more equations and operations, we could make $\vr w$ congruence
distributive or congruence modular. Our aim is to turn the chain
$J_1,\dots,J_{2k+1}$ into a longer chain of directed terms that ends at
something like $J_{2k+1}(x,z,z)$.

Notice that the operations of $\vr w$ are idempotent -- equivalently, if $\m
A\in \vr w$ then every one-element subset of $A$ is a subuniverse. Let $\m
F_3$ be the free algebra on three generators in $\vr w$, freely generated
(relative to $\vr w$) by the elements $x,y,z$. Let $\m F_2\leq \m F_3$ be
the subalgebra of $\m F_3$ freely generated by $x$ and $z$. Where feasible,
we will denote tuples of elements of $\m F_3$ without commas. For example,
we write $(xxz)$ rather than $(x,x,z)$. 

We shall be working with two binary relations $E,F$ on $\m F_2$. Define $\m
F$ to be the subalgebra of $\m f_2^2$ generated by the pairs $(x,x)$,
$(x,z)$ and $(z,z)$, that is
\[
  F = \big\{ ( t(xxz),t(xzz) ) 
    \colon \text{$t$ is a ternary term of $\vr w$} \big\}.
\]
Let $\vr j=\{J_1,\dots,J_{2k+1}\}$ and define $\vr g$ to be the set of all
$\vr w$-terms $t(x,y,z)$ such that whenever $\m A,\m B\in \vr w$ are
algebras such that $\m B\lhd^m_{\vr j} \m A$, then $\m B\lhd^m_{t}\m A$.
While the set $\vr g$ is hard to describe explicitly, one can easily see
that $\vr j\subseteq \vr g$ and that $\vr g$ is a subalgebra of $\m f_3$.

From this it immediately follows that 
\[
  E = \big\{ ( t(xxz),t(xzz) ) \colon t(x,y,z) \in \vr g \big\}
\]
is an admissible relation over $\m f_2$. Moreover, it is straightforward to
verify from the definition of absorption that $\vr g \lhd^m_{\vr j} \m f_3$,
from which it follows that $E\lhd^m_{\vr j} F$. We will view the pair
$E\lhd^m_{\vr j}F$ as a generic instance of absorption in $\vr w$. Notice
that $(x,x), (z,z)\in E$ since the projections $x, z$ belong to $\vr g$.
Thus, since all operations are idempotent, we have that the relations $E$
and $F$ are reflexive over $F_2$. That is, $(a,a)\in E$ for all $a\in F_2$.

We shall write $p\dashrightarrow q$ to indicate that the pair $(p,q)$
belongs to the transitive closure of $F$ and $p\rightarrow q$ to indicate
that $(p,q)$ belongs to the transitive closure of $E$. Both relations
$\dashrightarrow$ and $\rightarrow$ are admissible preorders of $\m F_2$
(i.e.\ they are transitive and reflexive). We leave it to the reader to
verify that $\rightarrow$ $\lhd^m_{\vr j}$ $\dashrightarrow$.

We now introduce \emph{left powers of elements} of $\m F_2$: for any
$a=a(x,z)\in \m F_2$ define $a^0=z$ and, inductively, 
\[
  a^{k+1}(x,z) = a(x,a^k).
\]
In more complicated expressions, we evaluate powers first, so for example
$a^2(b,c)$ means ``take $a^2(x,z)$ and substitute $x=b,z=c$'', giving us
$a(b,a(b,c))$. Observe that thus defined, exponentiation satisfies the
equalities $(a^k)^{\ell}=a^{k\ell}$ and $z^k=z$ for any $a\in \m F_2$ and
any $k,\ell$ nonnegative integers. 

Letting $J=J(x,z)=J_{2k+1}(x,z,z)$, we can state the core result of this
paper, whose proof takes up the remainder of this section.
\begin{thm}\label{thm:main4}
There exists $b\in \m F_2$ such that $x\rightarrow
J^{2^{k}}(b,J^{2^{k}-1})$.
\end{thm}

The next lemma is essential for our proof of Theorem~\ref{thm:main4}. Every
endomorphism of $\m F_2$ is uniquely determined by the elements to which it
sends $x$ and $z$, and, conversely, for any pair $a,b\in \m F_2$ there is an
endomorphism $\sigma$ of $\m F_2$ that sends each $c(x,z)\in \m F_2$ to
$c(a,b)=c(a(x,z),b(x,z))$ (in particular $\sigma(x)=a$ and $\sigma(z)=b$).
An endomorphism $\sigma$ of $\m F_2$ will be called \emph{special} if
$\sigma (x)\dashrightarrow \sigma(z)$.

\begin{lm}\label{lm:1}   
Every special endomorphism of $\m F_2$ respects $\dashrightarrow$ and
$\rightarrow$. That is, given  $a\dashrightarrow b$,
\begin{itemize}
  \item if $c=c(x,z)\dashrightarrow d(x,z) = d$ then $c(a,b)\dashrightarrow
    d(a,b)$; and
  \item if $c = c(x,z)\rightarrow d(x,z) = d$ then $c(a,b)\rightarrow
    d(a,b)$.
\end{itemize}
\end{lm}
\begin{proof}
To show that $\sigma$, moving $x$ to $a$ and $z$ to $b$ with
$a\dashrightarrow b$, respects $\dashrightarrow$, it suffices to show that
$c \Frel d$ implies $c(a,b)\dashrightarrow d(a,b)$. Let $c(x,z) \Frel
d(x,z)$. Thus there is a term $s(u,v,w)$ so that
\[
  c(xz)=s(xxz) 
  \qquad \text{and} \qquad
  d(xz)=s(xzz).
\]
Applying $\sigma$ to these equations, we have that
\[
  c(a,b)=s(a,a,b)
  \qquad \text{and} \qquad
  d(a,b)=s(a,b,b),
\]
or in a more suggestive matrix form:
\[
  \pmat{ c(a,b) \\
        d(a,b) }
  = s\pmat{ a & a & b \\
            a & b & b }.
\]
Now observe that in each of the three columns of the matrix on the right
hand side, the rows are related by $\dashrightarrow$. Since $s$ preserves
$\dashrightarrow$, we have $c(a,b)\dashrightarrow d(a,b)$, as required.

To show that $\sigma$ respects $\rightarrow$, it again suffices to show that
$c \Erel d$ implies $c(a,b)\rightarrow d(a,b)$. Let $c(x,z) \Erel d(x,z)$.
As before, there is a term $s(u,v,w)$ such that
\[
  c(xz)=s(xxz) 
  \qquad \text{and} \qquad
  d(xz)=s(xzz),
\]
but this time we also know that $s(x,y,z)\in \vr g$. We again apply $\sigma$
and write the result in a matrix form:
\[
  \pmat{ c(a,b) \\
        d(a,b) }
  = s\pmat{ a & a & b \\
            a & b & b }.
\]
Observe that in the first and third columns on the right hand side, the rows
are $\rightarrow$-related, while the middle column is
$\dashrightarrow$-related. Since $s\in \vr g$ it follows that
$\rightarrow\lhd^m_s \dashrightarrow$ and hence the pair on the left hand
side must be $\rightarrow$-related. Therefore $c(a,b)\rightarrow d(a,b)$, as
required.
\end{proof}  

Using Lemma~\ref{lm:1}, it is an easy exercise to show that if $a\rightarrow
b$, then $a^n\rightarrow b^n$ for any positive integer $n$.

\begin{df}\label{df:1}   
Let $n$ be a nonnegative integer. An \emph{$n$-fence} from $c$ to $d$,
denoted by $F(c,d)$, is a sequence of elements of $\m F_2$ satisfying
\[
  c 
  = a_0
  \rightarrow b_1
  \leftarrow a_1
  \rightarrow b_2
  \leftarrow a_2 
  \rightarrow \dots 
  \leftarrow a_n
  \rightarrow b_{n+1}
  = d.
\]

Let $n$ be a positive integer. An \emph{$n$-box $B$} is a sequence
$q_1\dashrightarrow p_1\dashrightarrow q_2\dashrightarrow p_2\dashrightarrow
q_3\dashrightarrow \dots \dashrightarrow q_n\dashrightarrow p_n$ such that
\[
  p_1 \rightarrow p_2 \rightarrow \dots \rightarrow p_n
  \qquad \text{and} \qquad 
  q_1 \rightarrow q_2 \rightarrow \dots \rightarrow q_n.
\]
An \emph{$n$-box} from $c$ to $b$ and $d$, denoted by $B(c;b,d)$, is an
$n$-box with $c= q_1$, $q_n\rightarrow b$, and $p_n\rightarrow d$. Note that
a $0$-fence from $c$ to $d$ is simply $c\rightarrow d$.
\end{df} 

\begin{figure} \centering
\begin{subfigure}{0.5\textwidth} \centering 
\begin{tikzpicture} [ font=\scriptsize, smooth, node distance=2em
                    , text height=1.5ex, text depth=0.25ex ]
  \node                          (a0) {$c=a_0$};
  \node [above=of a0,xshift=2em] (b1) {$b_1$};
  \node [below=of b1,xshift=2em] (a1) {$a_1$};
  \node [above=of a1,xshift=2em] (b2) {$b_2=d$};

  \foreach \a/\b in {a0/b1,a1/b1,a1/b2} {
    \draw[->] (\a) -- (\b);
  }
\end{tikzpicture} 
\caption{A $1$-fence $F(c,d)$.}
\end{subfigure}%
\begin{subfigure}{0.5\textwidth} \centering 
\begin{tikzpicture} [ font=\scriptsize, smooth, node distance=2em
                    , text height=1.5ex, text depth=0.25ex ]
  \node                          (a0) {$c=a_0$};
  \node [above=of a0,xshift=2em] (b1) {$b_1$};
  \node [below=of b1,xshift=2em] (a1) {$a_1$};
  \node [above=of a1,xshift=2em] (b2) {};

  \node [right=of b2]            (bn)  {};
  \node [below=of bn,xshift=2em] (an)  {$a_{n}$};
  \node [above=of an,xshift=2em] (bn1) {$b_{n+1}=d$};

  \coordinate (temp1) at ($(a0)!0.5!(b1)$);
  \coordinate (temp2) at ($(b2)!0.5!(bn)$);
  \node at (temp1-|temp2) {$\dots$};

  \foreach \a/\b in {a0/b1,a1/b1,a1/b2,an/bn,an/bn1} {
    \draw[->] (\a) -- (\b);
  }
\end{tikzpicture} 
\caption{An $n$-fence $F(c,d)$.}
\end{subfigure}
\caption{Pictures of fences.}
\end{figure}
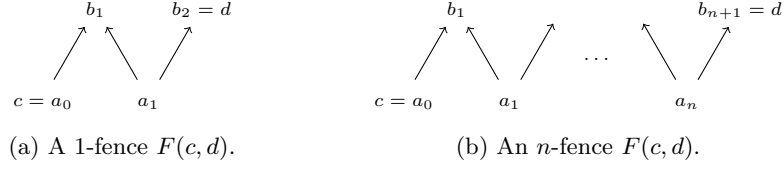 
\begin{figure} \centering
\begin{tikzpicture} [ font=\scriptsize, smooth, node distance=2em
                    , text height=1.5ex, text depth=0.25ex ]
  \node                          (c)  {$c=q_1$};
  \node [below=of c ,xshift=2em] (p1) {$p_1$};
  \node [above=of p1,xshift=2em] (q1) {$q_2$};
  \node [below=of q1,xshift=2em] (p2) {$p_2$};
  \node [above=of p2,xshift=2em] (q2) {$q_3$};
  \node [below=of q2,xshift=2em] (p3) {};
  \node [above=of p3,xshift=2em] (q3) {};

  \node [right=of p3,xshift=2em]  (pn2) {$\ $};
  \node [above=of pn2,xshift=2em] (qn1) {$\ $};
  \node [below=of qn1,xshift=2em] (pn1) {$p_{k}$};
  \node [above=of pn1,xshift=2em] (qn)  {$q_{k+1}$};
  \node [below=of qn,xshift=2em]  (pn)  {$p_{k+1}$};

  \node [right=of pn] (d) {$d$};
  \node [right=of qn] (b) {$b$};

  \coordinate (temp1) at ($(c)!0.5!(p1)$);
  \coordinate (temp2) at ($(q3)!0.5!(pn2)$);
  \node at (temp2) {$\dots$};

  \foreach \a/\b in
  {c/q1,q1/q2,p1/p2,q2/q3,p2/p3,qn1/qn,pn2/pn1,pn1/pn,pn/d,qn/b} {
    \draw[->] (\a) -- (\b);
  }
  \foreach \a/\b in
  {c/p1,p1/q1,q1/p2,p2/q2,q2/p3,p3/q3,pn2/qn1,qn1/pn1,pn1/qn,qn/pn} {
    \draw[->,dashed] (\a) -- (\b);
  }
\end{tikzpicture} 
\caption{A $(k+1)$-box $B(c;b,d)$.}
\label{fig:boxes}
\end{figure}
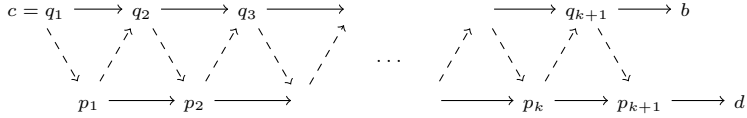 

The next three lemmas contain the heart of the proof of
Theorem~\ref{thm:main4}.

\begin{lm}\label{lm:2}   
Suppose that $B(c;b,d)$ is a $(k+1)$-box. Then $c\rightarrow
J_{2k+1}(b,d,d)$.
\end{lm}
\begin{proof} 
Label the vertices of the box from left to right according to
Figure~\ref{fig:boxes} as $q_1,p_1,q_2,p_2,\dots,q_{k+1},p_{k+1},b,d$.

Observe that since $q_1\rightarrow q_2$, $p_1\rightarrow p_2$, and
$q_1\dashrightarrow p_1\dashrightarrow q_2$ (and $\rightarrow \lhd^m_{\vr j}
\dashrightarrow$), we have the sequence: 
\[
  c 
  = J_1(q_1q_1p_1)
  \rightarrow J_1(q_2p_1p_1)
  = J_2(q_2p_1p_1)
  \rightarrow J_2(q_2q_2p_2)
  = J_3(q_2q_2p_2).
\]
Continuing in this vein, we obtain for $i$ ranging from 1 to $k$ the
sequence:
\begin{multline*}
  c
    \rightarrow J_{2i-1}(q_iq_ip_i)
    \rightarrow J_{2i-1}(q_{i+1}p_ip_i)
  = J_{2i}(q_{i+1}p_ip_i) \\
  \rightarrow J_{2i}(q_{i+1}q_{i+1}p_{i+1})
    = J_{2i+1}(q_{i+1}q_{i+1}p_{i+1}).
\end{multline*}
Letting $i=k$ (and thus $2i+1=n$), we conclude that 
\[
  c
  \rightarrow J_{2k+1}(q_{k+1},q_{k+1},p_{k+1})
  \rightarrow J_{2k+1}(q_{k+1},p_{k+1},p_{k+1}).
\]
Finally, using $q_{k+1}\rightarrow b$ and $p_{k+1}\rightarrow d$, we get
$c\rightarrow J_{2k+1}(b,d,d)$.
\end{proof}  

\begin{lm}\label{lm:3}   
Assume that there is a $1$-fence $x\rightarrow b\leftarrow a\rightarrow d$.
Then for every $\ell>1$ there is an $\ell$-box $B(x;b,d(b,d))$.
\end{lm}
\begin{proof} 
We put $q_1=x$ and $p_1=a(x,a)$. For $2\leq i\leq \ell$, let 
\[
  q_{i}=b(q_{i-1},a)
  \qquad \text{and} \qquad 
  p_{i}=a(q_i,a).
\]
We claim that the result is an $\ell$-box $B(x;b,d(b,d))$. The rest of the
proof consists of verifying the various $\dashrightarrow$ and $\rightarrow$
relations involved. We invite the reader to use Figure~\ref{figkbox} for a
reference (note that some diagonal edges are solid where the definition of a
box required only dashed edges -- this is all right since $\rightarrow$ is a
subset of $\dashrightarrow$).

\begin{figure} \centering
\begin{tikzpicture} [ font=\scriptsize, smooth, node distance=4em, 
  text height=1.5ex, text depth=0.25ex ]

  \coordinate                          (q1);
  \coordinate [below=of q1,xshift=3.5em] (p1);
  \coordinate [above=of p1,xshift=3.5em] (q2);
  \coordinate [below=of q2,xshift=3.5em] (p2);
  \coordinate [above=of p2,xshift=3.5em] (q3);

  \coordinate [right=of p2,  xshift=3.5em] (pk-1);
  \coordinate [above=of pk-1,xshift=3.5em] (qk);
  \coordinate [below=of qk,  xshift=3.5em] (pk);
  \coordinate [right=of qk]                (b);
  \coordinate [right=of pk]                (dbd);

  \node (q1) at (q1) {$x=q_1$};
  \node (p1) at (p1) {$p_1=a(x,a)$};
  \node (q2) at (q2) {$q_2=b(x,a)$};
  \node (p2) at (p2) {$p_2=a(q_2,a)$};
  \node (q3) at (q3) {$q_3=b(q_2,a)$};
  \node (pk-1) at (pk-1) {$p_{\ell-1}$};
  \node (qk) at (qk) {$q_\ell$};
  \node (pk) at (pk) {$p_\ell$};
  \node (b) at (b) {$b$};
  \node (dbd) at (dbd) {$d(b,d)$};

  \node at ($(q3)!0.5!(pk-1)$) {$\dots$};

  \foreach \a/\b in 
  {q1/q2,p1/p2,q2/q3,p1/q2,p2/q3,pk-1/qk,qk/b,pk-1/pk,pk/dbd} {
    \draw[->] (\a) -- (\b); 
  }
  \foreach \a/\b in {q1/p1,q2/p2,qk/pk} {
    \draw[->,dashed] (\a) -- (\b);
  }
\end{tikzpicture} 
\caption{The $\ell$-box $B(x;b,d(b,d))$.}
\label{figkbox}
\end{figure}
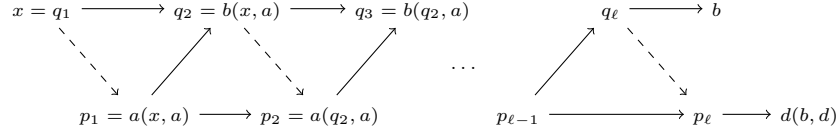 

Observe that $x\dashrightarrow a$, so the endomorphism $\sigma$ sending $x$
to $x$ and $z$ to $a$ is special. It is easy to see that $\sigma(b)=b(x,a)=
q_2$ and $\sigma(a)=a(x,a)=p_1$. Since $x\rightarrow b \leftarrow a$, it
follows by Lemma~\ref{lm:1} that $x\rightarrow q_2\leftarrow p_1$.

We now proceed by induction to prove that $q_i\rightarrow q_{i+1}$ and
$p_i\rightarrow p_{i+1}$ for all $i=1,\dots,\ell-1$. We already know the
arrows for $i=1$, and from $q_{i-1}\rightarrow q_i$, we easily get both
$q_i=b(q_{i-1},a) \rightarrow b(q_i,a)=q_{i+1}$ and $p_i\rightarrow p_{i+1}$
for all applicable values of $i$.

Observe that $q_1=x\dashrightarrow a$. Since $q_{i}=b(q_{i-1},a)$, induction
gives us that $q_i\dashrightarrow a$ for all $i$. Repeated use of this set
of dashed arrows allows us to prove that $p_i\rightarrow q_{i+1}$ and
$q_{i}\dashrightarrow p_i$ for all $i$ in the following way. Consider first
the endomorphism $\sigma$ sending $x$ to $q_i$ and $z$ to $a$. Since
$q_i\dashrightarrow a$, this is a special endomorphism. Since $a\rightarrow
b$, we have $p_i=\sigma(a)\rightarrow \sigma(b)=q_{i+1}$ for all $i$. To see
$q_i\dashrightarrow p_i$, observe that $q_i=a(q_i,q_i)\dashrightarrow
a(q_i,a)=p_i$.

All that remains now is to get the two arrows at the rightmost end of the
box. Similarly to the previous paragraph, it is easy to prove by induction
on $i$ that $q_i\rightarrow b$ for all $i$, so in particular
$q_{\ell}\rightarrow b$. To obtain $p_{\ell}\rightarrow d(b,d)$, observe
that $p_{\ell}=a(q_{\ell},a)\rightarrow d(q_{\ell},a)\rightarrow
d(b,a)\rightarrow d(b,d)$ (we have used first Lemma~\ref{lm:1}, then
$q_{\ell}\rightarrow b$, and finally $a\rightarrow d$).
\end{proof}  

\begin{lm}\label{lm:5}   
For each $0\leq i < k$, there exists a $(k-i)$-fence from $x$ to
$J^{2^{i+1}-1}$. (Recall that $J=J(x,z)=J_{2k+1}(x,z,z)$.)
\end{lm}
\begin{proof} 
We proceed by induction on $i$. For $i=0$, we get a $k$-fence from $x$ to
$J$ by putting $b_{\ell}=J_{2\ell-1}(x,z,z)$ and $a_{\ell} =
J_{2\ell}(x,x,z)$, for $1\leq \ell\leq k$.

Suppose now that $1\leq i < k$ and we have a $(k-i+1)$-fence 
\begin{equation} \label{eqn:ind-fence}
  x
  \rightarrow b_1
  \leftarrow a_1
  \rightarrow b_2
  \leftarrow a_2
  \rightarrow \dots
  \leftarrow a_{k-i}
  \rightarrow b_{k-i+1}
  \leftarrow a_{k-i+1}
  \rightarrow J^{2^{i}-1}.
\end{equation}
We proceed to construct a $(k-i)$-fence from $x$ to $J^{2^{i+1}-1}$.

Applying first Lemma~\ref{lm:3} and then Lemma~\ref{lm:2} to the 1-fence
with vertices $x,b_1,a_1,b_2$ above, we get 
\[
  x
  \rightarrow J_{2k+1}(b_1,b_2(b_1,b_2),b_2(b_1,b_2))
  = J(b_1,b_2(b_1,b_2)).
\]
Denote the term on the right hand side of the above arrow by $b_1'$. Using
$b_1\leftarrow x$, we get
\[
  b_1'
  = J(b_1,b_2(b_1,b_2))
  \leftarrow J(x,b_2(x,b_2))
  = J(x,b_2^2).
\]
Since $b_2^2\leftarrow a_2^2$, we obtain $b_1'\leftarrow J(x,a_2^2)$.
Consider the sequence $b_1', a_1' = J(x,a_2^2)$, and 
\[
  a_\ell'=J(x,a_{\ell+1}^2)
  \qquad \text{and} \qquad 
  b_\ell'=J(x,b_{\ell+1}^2)
\]
for $2\leq \ell \leq k-i$. It is easy to verify that
\[
  x
  \rightarrow b_1'
  \leftarrow a_1'
  \rightarrow b_2'
  \leftarrow a_3' 
  \rightarrow \dots
  \leftarrow a_{k-i}'.
\]

Let us look at the element $a_{k-i}'$ in this fence. We have
\[
  a_{k-i}'
  = J(x,a_{k-i+1}^2)
  \rightarrow J\left(x,\left(J^{2^{i}-1}\right)^2\right)
  = J(x,J^{2^{i+1}-2})=J^{2^{i+1}-1}
\]
(we use $a_{k-i+1} \rightarrow J^{2^i-1}$ from~\eqref{eqn:ind-fence} above).
We have therefore found a $(k-i)$-fence from $x$ to $J^{2^{i+1}-1}$, as was
needed.
\end{proof}  

We are now ready to prove Theorem~\ref{thm:main4}.
\begin{thm*}[Theorem~\ref{thm:main4}]  
There exists $b\in \m F_2$ such that $x\rightarrow
J^{2^{k}}(b,J^{2^{k}-1})$.
\end{thm*}
\begin{proof}
By taking $i=k-1$ in Lemma~\ref{lm:5}, we obtain a 1-fence $x\rightarrow b
\leftarrow a \rightarrow J^{2^{k}-1}$. Applying Lemmas~\ref{lm:3} and
\ref{lm:2}, and observing that
\[
  J(b,J^{2^{k}-1}(b,J^{2^{k}-1}))
  = J^{2^{k}}(b,J^{2^{k}-1}),
\]
we get $x\rightarrow J^{2^{k}}(b,J^{2^{k}-1})$.
\end{proof}  

\section{Directed J\'onsson terms}\label{section4} 
\begin{thm*}[Theorem~\ref{thm:main3}]   
Let $\vr v$ be a variety, and $\vr j$ be a chain of weak J\'onsson terms of
$\vr v$. Then there exists a chain $\vr {d}$ of weak directed J\'onsson
terms of $\vr v$ such that for all $\m A,\m B\in \vr v$ we have $\m
B\lhd^m_{\vr j} \m A$ implies $\m B \lhd^m_{\vr{d}} \m A$.
\end{thm*}
\begin{proof}
Let $J_1,\dots,J_{2k+1}$ be a chain of weak J\'onsson terms in $\vr v$.  By
taking an inessential expansion of $\vr v$, we can assume that $J_i$ are
basic operations of $\vr v$. Consider the variety $\vr w$ from the previous
chapter. Since the equational basis of $\vr w$ is a subset of the identities
true in $\vr v$, the variety $\vr w$ interprets into $\vr v$.

Theorem~\ref{thm:main4} gives us that there is a chain $\vr d=\{D_1,\dots,
D_m\}\subseteq \vr g$ such that the system of equalities
\begin{align*}
  & D_1(xxz) = x, \\
  & D_i(xzz) = D_{i+1}(xxz) \quad \text{for each } i= 1,\dots,m-1, \\
  & D_m(xzz) = J^{2^{k}}(b,J^{2^{k}-1})
\end{align*}  
holds in $\vr w$. Since $\vr w$ interprets into $\vr v$, these equalities
must also hold in $\vr v$. Moreover, in $\vr v$ we have the equality
$J(x,z)=J_{2k+1}(x,z,z)= z$, so $J^{2^{k}}(b,J^{2^{k}-1})= z$. 

Finally, let $\m B \leq \m A$ be algebras in $\vr V$. By removing all of the
basic operations except $J_1,\dots, J_{2k+1}$, we obtain a pair of reducts
$\m B^\star\leq \m A^\star$ which both lie in $\vr W$. If $\m B \lhd^m_{\vr
j} \m A$, then trivially $\m B^\star \lhd^m_{\vr j}\m A^\star$, and
$D_1,\dots,D_m\in \vr g$ gives us $\m B^\star \lhd^m_{\vr{d}} \m A^\star$.
Since $\m A^\star$ is a reduct of $\m A$, we immediately have $\m B
\lhd^m_{\vr{d}} \m A$.

The chain $D_1,\dots,D_m$ middle absorbs anything that $\vr j$ absorbs, and
satisfies in $\vr v$ the system of equalities
\begin{align*}
  & D_1(xxz) = x, \\
  & D_i(xzz) = D_{i+1}(xxz) \quad \text{for each } i=1,\dots,k-1, \\
  & D_m(xzz) = z.
\end{align*}
Therefore, $D_1,\dots,D_m$ is the weak directed J\'onsson chain $\vr d$ we
were looking for.
\end{proof}  

\begin{cor}\label{cor:J}   
Let $\vr v$ be a variety with a system of J\'{o}nsson terms $\vr j$. Then
$\vr v$ has a system of directed J\'{o}nsson terms.
\end{cor}
\begin{proof}
Let $\m F_{3}^{id}$ be idempotent reduct of the free three generated algebra
in $\vr v$. Then $\m F_{3}^{id}$ contains a chain of J\'onsson terms $\vr j$
such that $\{x\}\lhd^m_{\vr j} F_{3}^{id}$. Applying Theorem~\ref{thm:main3}
with $\m B = \{x\}$ and $\m A = \m F_{3}^{id}$ gives us that there is a
chain of directed weak J\'onsson terms $\vr d$ such that $\{x\}\lhd^m_{\vr
D} F_3^{id}$. Every $D_i$ in $\vr d$ satisfies $D_i(x,y,x)= x$, making
$\vr d$ a chain of directed J\'onsson terms for $\vr v$.
\end{proof}  

We are now ready to give a full proof of Theorem~\ref{thm:preorders}.
\begin{thm*}[Theorem~\ref{thm:preorders}]   
Suppose that $E$ and $F$ are admissible preorders on $\m a$ (that is, they
are subalgebras of $\m a^2$ that are reflexive and transitive). If
$E\lhd_{J} F$ then $E=F$.
\end{thm*}
\begin{proof}
Let $\m a\in \vr v$, where $\vr v$ has a weak J\'{o}nsson system of terms
$\vr j$ and suppose that $E,F$ are admissible preorders of $\m a$ with
$E\lhd_{\vr j}F$. Let $\vr d= \{D_1,\dots,D_m\}$ be the system of weak
directed J\'{o}nsson terms for $\m A$ supplied by Theorem~\ref{thm:main3}.
Then $E\lhd^m_{\vr d}F$ and so for every $(a,b)\in F$ we have:
\begin{multline*}
  a
    = D_1(a,a,b)
    \,E\, D_1(a,b,b)
    = D_2(a,a,b)
    \,E\, D_2(a,b,b)
    = \cdots \\
  \dots = D_m(a,a,b)
    \,E\, D_m(a,b,b)
    = b,
\end{multline*}
yielding $(a,b)\in E$.
\end{proof}  

\section{Pixley terms}\label{section6}   
We now proceed to prove the statement (\ref{nCPpart}) of
Theorem~\ref{thm:main1} (Lipparini's Proposition 5 in~\cite{lipparini}).

\begin{thm}\label{thm2}   
Let $k$ be any positive integer and let $\vr v$ be a $(k+1)$-permutable
variety with a system of J\'{o}nsson terms $\vr j$. Then $\vr v$ has a
system of Pixley terms $\vr p=\{P_1,\dots,P_k\}$ such that whenever $\m A,\m
B\in \vr v$ and $\m B\lhd^m_{\vr j} \m A$, then $\m B\lhd^m_{\vr p}\m A$.
\end{thm}
\begin{proof}
The proof is a variant of the proof of Theorem~\ref{thm:main4}. Choose and
fix an arbitrary idempotent variety $\vr v$ that has a system $\vr j$ of
J\'{o}nsson terms and a system $H_1,\dots,H_k$ of Hagemann-Mitschke terms,
i.e.\ terms that satisfy the equations 
\begin{align*}
  & H_1(x,z,z)=x, \qquad H_k(x,x,z)=z, \\
  & H_i(x,x,z)=H_{i+1}(x,z,z) \qquad \text{for } 1\leq i<k.
\end{align*}
Starting as in Section~\ref{section3}, we let $\m F_2$ be the free algebra
of rank two in $\vr v$ freely generated by $x$ and $z$. Let $F$ be the
subalgebra of $\m f_2^2$ generated by the pairs $(x,x)$, $(x,z)$, and
$(z,z)$, that is
\[
  F = \big\{ ( t(xxz),t(xzz) ) \colon t \text{ a term of } \vr v \big\}.
\]
As before, we define $\vr g$ to be the set of all $\vr v$-terms $t(x,y,z)$
such that whenever $\m A,\m B\in \vr v$ and $\m B\lhd^m_{\vr j} \m A$, then
also $\m B\lhd^m_{t} \m A$, and let
\[
  E = \big\{ ( t(xxz),t(xzz) ) \colon t(xyz) \in \vr g \big\}.
\]
As before, $E$ and $F$ are idempotent admissible relations over $\m f_2$ and
we have $E\lhd^m_{\vr j} F$.

Using $p\rightarrow q$ to denote that $(p,q)$ belongs to the transitive
closure of $E$, we proved in Sections~\ref{section3} and \ref{section4} that
$x\rightarrow z$. Since the operations $H_i$ respect $E$ and $\rightarrow$,
we have that $z\rightarrow x$. This is a classical observation, but the
proof is easy and so we give it in the following paragraph.

Since $x\rightarrow x$, $z\rightarrow z$, and $x\rightarrow z$, we have
\begin{multline*}
  z 
    = H_1(z,x,x)
    \rightarrow H_1(z,z,x)
    = H_2(z,x,x) \\
  \rightarrow H_2(z,z,x)
    = H_3(z,x,x)
    \rightarrow \dots
    \rightarrow H_k(z,z,x)
    = x.
\end{multline*}
Transitivity of $\rightarrow$ gives $z\rightarrow x$.

We now demonstrate the classical fact that $E^{k+1}=E^k$, which gives us
that $E^k$ is the transitive closure of $E$ (and in particular $(z,x)\in
E^k$). Since $E$ is reflexive, we have $E^k\usub E^{k+1}$. Suppose that we
have $(a,b)\in E^{k+1}$. Then there are $a_i$ for $i\leq k+1$ such that 
\[
  a
  = a_0
  \Erel a_1
  \Erel \cdots 
  \Erel a_i
  \Erel a_{i+1}
  \Erel \cdots 
  \Erel a_{k+1}
  = b.
\]
Letting $c_i=H_{i+1}(a_i,a_{i+1},a_{i+1})$ for $0\leq i<k$, it is easy to
verify that
\[
  a 
  \Erel c_1 
  \Erel c_2 
  \Erel \cdots
  \Erel c_{k-1} 
  \Erel b,
\]
so $(a,b)\in E^k$.

Continuing with the main proof, we have $(z,x)\in E^k$. This means that
there are $\vr v$-terms 
\[
  D_1(x,y,z),\dots, D_k(x,y,z)\in \vr g
\]
satisfying $z=D_1(x,x,z)$, $D_i(x,z,z)=D_{i+1}(x,x,z)$ for $1\leq i<k$, and
$D_k(x,z,z)=x$. As before, whenever $\m A\lhd^m_\vr j \m B$, we also have
$\m A\lhd^m_{D_1,\dots,D_k} \m B$, so in particular $D_i(x,z,x)=x$. The
terms $P_i=D_{k-i+1}$ for $1\leq i\leq k$ then satisfy $\Pchain(k)$ in $\vr
v$.
\end{proof}  

\section{Directed Gumm terms}\label{section5}  
To conclude the proof of Theorem~\ref{thm:main1}, we focus on Gumm terms and
introduce Gumm absorption. Gumm terms $\Gchain(n)$, directed Gumm terms
$\DGchain(n)$, and weak versions of both were defined in the introduction.
(Recall that weak versions drop the conditions $J_i(x,y,x)=x$.) 

When a variety $\vr v$ has a chain of weak Gumm terms (respectively, weak
directed Gumm terms) $J_1,\dots, J_n, P$, and $\m a,\m b\in \vr v$ are such
that $\m b\leq \m a$, we say that $\m B$ \emph{Gumm absorbs} $\m A$
(respectively, \emph{directed Gumm absorbs} $\m A$) with respect to these
chains if $\m B\lhd^m_{J_1,\dots,J_n} \m A$. We next state and prove a
variant of Theorem~\ref{thm:main3} for Gumm terms.

\begin{thm}\label{thm1}   
Let $\vr v$ be a variety, and $J_1,\dots,J_{2k+1},P$ be a chain of weak Gumm
terms of $\vr v$. Then there exists a chain $D_1,\dots,D_m,Q$ of weak
directed Gumm terms of $\vr v$ such that whenever $\m A,\m B\in \vr v$ and
$\m B\lhd^m_{\vr J_1,\dots,J_{2k+1}} \m A$, then $\m B\lhd^m_{\vr
D_1,\dots,D_m} \m A$.

In particular, if $J_1,\dots,J_{2k+1}, P$ is a chain of Gumm terms, then
it follows that $D_1,\dots,D_m, Q$ is a chain of directed Gumm terms.
\end{thm}

Note that Theorem~\ref{thm1} immediately gives us the third assertion of
Theorem~\ref{thm:main1}.

\begin{proof}
The argument follows the same pattern as our proof of
Theorem~\ref{thm:main3}. We consider the variety $\vr w$ and use
Theorem~\ref{thm:main4} to obtain terms $D_1,\dots,D_m$ in $\vr v$ such that
\begin{align*}
  & D_1(xxz) = x, \\
  & D_i(xzz) = D_{i+1}(xzz) \quad \text{for each } i=1,\dots,m-1, \\
  & D_m(xzz) = J^{2^{k}}(b,J^{2^{k}-1}),
\end{align*} 
where $b$ is some term composed from $J_1,\dots,J_{2k+1}$, and
$J(x,y)=J_{2k+1}(xyy)= P(xyy)$ in $\vr v$. The term
$J^{2^{k}}(b,J^{2^{k}-1})$ can be expressed as
\[
  \underbrace{ J(b,J(b,\dots,J }_{ \text{$2^k$-many $J$'s} }
    (b,\underbrace{ J(x,J(x,\dots,J }_{ \text{$(2^k-1)$-many $J$'s} }(x,z) 
    ) \dots )) \dots )),
\]
More formally, if we let $d_0(x,z)=z$ and
\begin{align*}
  d_i(x,z) & = J_{2k+1}(x,d_{i-1}(x,z),d_{i-1}(x,z)) && \text{for } 1\leq i< 2^k, \\
  d_i(x,z) & = J_{2k+1}(b(x,z),d_{i-1}(x,z),d_{i-1}(x,z)) && \text{for } 2^k\leq i<2^{k+1},
\end{align*}
then we will have $ d_{2^{k+1}-1}(x,z)=D_m(xzz)$.

Now we systematically rewrite $J^{2^{k}}(b,J^{2^{k}-1})$, replacing all but
the rightmost occurrence of $z$ by $y$, and replacing all occurrences of
$J_{2k+1}$ by $P$, to obtain a term $Q(xyz)$. 

More formally, we let $Q_0(xyz)=z$, $Q_1(x,y,z)=P(x,y,z)$, and 
\[
  Q_i(x,y,z) = P(x,Q_{i-1}(x,y,y),Q_{i-1}(x,y,z))
\]
for $2\leq i<2^k$, and
\[
  Q_i(x,y,z) = P(b(x,y),Q_{i-1}(x,y,y),Q_{i-1}(x,y,z))
\]
for $2^k\leq i<2^{k+1}$. Having done that, we let
$Q(x,y,z)=Q_{2^{k+1}-1}(x,y,z)$.

Using the equality $J_{2k+1}(xzz)= P(xzz)$, one can easily prove that
$Q(xzz)=J^{2^k}(b, J^{2^k-1})$ in $\m A$. Idempotence of $b$ together with
$P(xxz)=z$ then gives us that $Q(xxz)=z$.

Thus we have a chain of weak directed Gumm terms $D_1,\dots,D_m,Q$. Since
$D_1,\dots, D_m\in \vr g$, the middle absorption property follows as in
Theorem~\ref{thm:main3}. Showing that ordinary Gumm terms imply the
existence of a chain of directed Gumm terms is analogous to the proof of
Corollary~\ref{cor:J}.
\end{proof}  

We can now also state and prove a version of Theorem~\ref{thm:preorders} for
Gumm terms.

\begin{thm}\label{thm0}   
Suppose that $\m E$ and $\m F$ are reflexive subalgebras of $\m a^2$ and
that $\m E$ Gumm absorbs $\m F$. Whenever $(a,b)\in F$, there is $c\in A$
such that $(b,c)\in F$ and $(a,c)$ belongs to the transitive closure of $E$.
\end{thm}
\begin{proof}
Apply Theorem~\ref{thm1} to get weak directed Gumm terms $D_1,\dots,D_m,Q$
for the variety generated by $\m a$ so that $\m E\lhd_{D_1,\dots,D_m}\m F$.
Then
\[
  a
  \,E\, D_1(a,b,b)
  \,E\, D_2(a,b,b)
  \,E\, \cdots 
  \,E\, D_m(a,b,b)
  = Q(a,b,b)
  = c,
\]
where $b = Q(a,a,b)\,F\, Q(a,b,b) = c$.
\end{proof}  
 
\section{Final Remarks}   
We have worked through the various parts of the proof of
Theorem~\ref{thm:main4}, calculating the precise lengths of the $E$-chains
produced. The final formula for the length of the $E$-chain connecting $x$
to $J^{2^{k}}(b,J^{2^{k}-1})$ simplifies to
\[
  \frac{(2k+1)(k+1)((k+1)^{k-2}-1)}{k}.
\]
Thus, to be precise, we have proved that $J(k)$ implies $DJ(m)$ with $m$
equal to the displayed number. This is our best value for $m$. It would be
interesting to know if a different approach, or the introduction of some new
tricks, can lower this value of $m$ substantially. We close the paper by
posing a problem stemming from our work here.

\begin{prob}
Does there exist a sequence of algebras $\m a_1,\m a_2,\dots$ such that each
$\m a_n$ is $\Jchain (n)$, but the least $m$ such that $\m a_n$ is
$\DJchain(m)$ grows at least exponentially in $n$? 
\end{prob}

\bibliographystyle{plain}   
\end{document}